\documentclass{article}
\usepackage[english]{babel}
\usepackage[utf8]{inputenc}
\UseRawInputEncoding
\usepackage{setspace}

\usepackage{amssymb}
\usepackage{graphicx,color} 
\usepackage{epsfig}
\usepackage{titlesec}
\usepackage{sectsty}
\usepackage{a4,amsfonts,latexsym,amscd,amsopn,amssymb,amsmath, amsthm,bbm,stmaryrd,wasysym,mathbbol}
\usepackage{amstext}
\usepackage{verbatim}
\usepackage{catchfile}
\usepackage{autobreak}
\usepackage{hyperref}

\sectionfont{\fontsize{12}{15}\selectfont}

\newtheorem{Theorem}{Theorem}[section]

\newtheorem{lemma}[Theorem]{Lemma}

\theoremstyle{definition}

\newcommand\scalemath[2]{\scalebox{#1}{\mbox{\ensuremath{\displaystyle #2}}}}

\providecommand{\keywords}[1]
{
  \small	
  \textit{Keywords:} #1
}
\providecommand{\subjclass}[2][2010]
{
\small
\emph{AMS Mathematics Subject Classification 2010:} #2
}

\title{Diffusion Approximation for Transport Equations with Dissipative Drifts for Time Dependent Coefficients}

\author{First Author Name$^{a}$$^{*}$, Second Author Name$^{b}$$^{c}$, etc.$^{a}$$^{c}$ \\
       \small $^{a}$Department, University, City, Country \\
        \small $^{b}$Department, University, City, Country \\
        \small $^{c}$Department, University, City, Country \\\\
        \small $^{*}$Corresponding author: first name, initials, 
}
\author{Luca Di Persio \footnote{College of Mathematics, 
Department of Computer Science, University of Verona, strada le Grazie 15 - 37134 Verona - Italy, luca.dipersio@univr.it}
  \and  Yuri Kondratiev \footnote{Dragomanov University, Kiev, yukondrat@gmail.com}
  \and Viktorya Vardanyan \footnote{Department of Mathematics, University of Trento, Via Sommarive 14-38123 Povo(TN)-Italy, viktorya.vardanyan@unitn.it}
}

\date{} 

\begin{document}
\maketitle

\begin{abstract} 
We consider a generalization of classical results of Freidlin and  Wentzell  to the case of time dependent dissipative drifts.
We show the convergence of diffusions with  multiplicative noise in the zero limit of a diffusivity parameter to the related dynamical systems. The solution to the associated transport equation is obtained as an application.

\end{abstract} \hspace{10pt}

\keywords{diffusion process, dynamical system, perturbation parameter,  local Lipschitz condition, time dependent dissipative drift, zeroth order approximation,  Cauchy problem, transport equation}\\
\subjclass[2010]{47D07, 37P30, 60J65, 60G55}

\section{Introduction}
Dynamical systems with small random perturbations are the subject of well know studies by Freidlin and Wentzell \cite{FW}.
Their classical results essentially cover the case of diffusions with regular drifts and small parameters in multiplicative noise. The main problem within latter scenario, which is the focus of our study,   is the convergence of such diffusions to deterministic dynamical systems and related evolution equation of transport type.

We study stochastic differential equations(SDEs) with a small perturbation parameter, in which the coefficients depend on the time.
The dependence on time is the  main novelty in our considerations comparing with \cite{FW}. 

Let us also underline that in \cite{FW}
the authors work under standard probabilistic assumptions on the drift in the considered diffusions, namely: global Lipschitz
condition and linear bound on the growth. Nevertheless, in several applications in mathematical physics, these assumptions are too restrictive, as in the case of, e.g., the stochastic quantisation method. 

 We consider the  dissipative condition on the drift coefficient and the local Lipschitz condition on the drift and diffusion coefficients  formulated for $0\leq t \leq T$, where $0<T<+\infty$ is a finite time horizon, with constants depending on $T$. Under latter conditions,  we prove both existence and uniqueness for the perturbed SDE. Moreover, we show the convergence result for the solution of the  perturbed system to the solution of the unperturbed system, when the perturbation parameter approaches zero. We then consider an application of the above-mentioned results to the Cauchy problem and the transport equations.

\section{The model}
Let $(\Omega, \mathcal{F},\{\mathcal{F}_t\}_{t\in [0,T]}, \mathbb{P})$ be a reference filtered probability space. For the sake of completeness, let us recall that the latter means that $\Omega$ is a nonempty set,  interpreted as the space of elementary events, while $\mathcal{F}$, is a $\sigma$-algebra of subsets of  $\Omega$, $\mathbb{P}$ is a probability measure on the $\sigma$-algebra $\mathcal{F}$, and
$\{\mathcal{F}_t\}_{t \in [0,T]}$, $0<T<+\infty$ being a finite horizon time, is a filtration over $(\Omega, \mathcal{F}, \mathbb{P})$. Finally, we define $W=\{W_t\}_{t \in [0,T]}$ as a given $l$-dimensional standard Brownian motion adapted to the just defined filtration.\\
We consider an equation
\begin{equation}
\label{ode}
\frac{d}{dt} x(t)=b(t,x(t)),\ x(0)=x_0
\end{equation}
and the perturbed stochastic differential equation
\begin{equation}
    \label{perturbed sde}
    dX^\varepsilon(t)= b(t,X^\varepsilon(t))dt + \varepsilon \sigma(t,X^\varepsilon(t)) dW_t, \ X^\varepsilon(0) =x_0
\end{equation}
in $\mathbb{R}^r$, $\varepsilon$ representing a fixed, {\it small}, multiplicative constant. The functions $b(t,x)$ and $\sigma(t,x)$ are measurable functions on $[0,T]\times \mathbb{R}^r$, taking,  respectively, values in $\mathbb{R}^r$ and $\mathbb{R}^{r\times l}$, 
$x_0$ is a given $\mathbb{R}^r$-valued random variable independent of $W_t$.\\
The stochastic process $X^\varepsilon(t)$ is considered as a result of small perturbations of \eqref{ode}, with perturbation parameter $\varepsilon$.
Solving \eqref{perturbed sde}, means finding a random process $X_t = X_t(W)$ which satisfies the relation
\begin{equation*}
    X^\varepsilon(t)=x_0+\int_0^t b(s,X^\varepsilon(s))ds+\varepsilon \int_0^t\sigma(s,X^\varepsilon(s))dW_s,
\end{equation*}
with probability 1 for every $t \in [0,T]$.
Let us first assume that the coefficients of our diffusion are Lipschitz continuous and with linear growth bound. This implies that \eqref{perturbed sde} admits a uniquely determined solution. Equations of the form \eqref{perturbed sde} were studied, e.g.,
by Gihman and Skorohod \cite{GS 3} in Section 2 of Chapter 2 and in Section 3 of Chapter 3. \\
We want to prove that such (sequence of) solution(s) of SDE (\ref{perturbed sde}) converges to the solution of (\ref{ode}) uniformly, when $\varepsilon \rightarrow 0^+$, also estimating the convergence rate.

Before proving main results, let us recall the following  result that will be then exploited in our proofs.
\begin{lemma} \label{Gronwall-Lemma} [Gronwall] 
	Let $m(t), t \in [0,T]$, be a non-negative function satisfying  
	\begin{equation}
	\label{Gronwall}
	m(t) \leq C + \alpha \int_{0}^{t}m(s)ds, \ t\in [0,T]
	\end{equation}
	with $C,\alpha \geq 0$. \\
	Then $m(t) \leq C e^{\alpha t},$ for $ t \in [0,T]$.
\end{lemma}

\section{Zeroth Order Approximation}

\begin{Theorem} \label{Theorem 1.2}
	Assume that the coefficients of \eqref{perturbed sde} are linearly bounded and satisfy a uniform Lipschitz condition, i.e. for every $T>0$ there exist  constants $K_T$, $L_T$ such that for all $x,y \in \mathbb{R}^r$, $t \in [0,T]$:
	\begin{equation}
	    \label{boundedness}
	    |b(t, x)|^2+|\sigma(t, x)|^2 \leq K_T^2(1+|x|^2)
	\end{equation}
     \begin{equation}
	    \label{Lipschitz}
	    |b(t, x)-b(t, y)|^2+|\sigma(t, x)-\sigma(t, y)|^2 \leq L_T^2|x-y|^2
	\end{equation}
	Then for all $t>0$ and $\delta>0$ we have:
	\begin{enumerate}
		\item $E|X^{\varepsilon}(t)-x(t)|^2\leq \varepsilon^2 a(t)$, and
		\item $\lim_{\varepsilon \to 0} P\lbrace  \max_{0\leq s \leq t}|X^{\varepsilon}(s)-x(s)| > \delta \rbrace=0$
	\end{enumerate}
where a(t) is a monotone increasing function expressed in terms of $E|x_0|$, $L_T$ and $K_T$.
	
\end{Theorem}

\begin{proof}
We start showing that $1+E|X^\varepsilon(t)|^2$ is bounded uniformly in $\varepsilon \in [0,1]$. To this end, we use Ito's formula, the Cauchy-Schwarz inequality and the condition (\ref{boundedness})  of linear boundedness for the SDE (\ref{perturbed sde}) coefficients.
\begin{align*}
	  1+ E|X^\varepsilon(t)|^2 
	  &= 1+E|x_0|^2+2\int_0^t E < X^\varepsilon(s),b(s,X^{\varepsilon}(s))> ds +
	   \varepsilon^2\int_0^t E |\sigma(s,X^{\varepsilon}(s))|^2ds\\
	  & \leq 1+E|x_0|^2+2\int_0^t E\sqrt{|X^\varepsilon(s)|^2  b(s,X^{\varepsilon}(s))^2} ds +
	   \varepsilon^2\int_0^t E |\sigma(s,X^{\varepsilon}(s))|^2ds \\
	  & \leq 1+E|x_0|^2+2\int_0^t E\sqrt{|X^\varepsilon(s)|^2 K_T^2(1+|X^\varepsilon(s)|^2)} ds\\
   &+ \varepsilon^2\int_0^t E [K_T^2(1+|X^\varepsilon(s)|^2)]ds \\
	  & \leq 1+E|x_0|^2 +(2K_T+\varepsilon^2K_T^2)\int_0^t(1+E|X^\varepsilon(s)|^2)ds
\end{align*}

Then, in view of applying the Gronwall Lemma (\ref{Gronwall}), we choose $m(t)=1+E|X^\varepsilon(t)|^2$, $C=1+E|x_0|^2 $ and $\alpha=2K_T+\varepsilon^2K_T^2$, then
	obtaining 
	\begin{equation}
	\label{(1.5)}
	1+E|X^\varepsilon(t)|^2\leq (1+E|x_0|^2) \exp[(2K_T+\varepsilon^2K_T^2)t].
	\end{equation}
	
In what follows, we use the above result to show that $E|X^\varepsilon(t)-x(t)|^2\leq \varepsilon^2 a(t)$. 
Proceeding as before, using Lipschitz condition \eqref{Lipschitz} and the boundedness of $1+E|X^\varepsilon(t)|^2$ i.e.\eqref{(1.5)}, we have:
	\begin{align*}
	  E|X^\varepsilon(t)-x(t)|^2 
	  &\leq 2\int_0^t E < X^\varepsilon(s)-x(s),b(s,X^{\varepsilon}(s))-b(s,x(s))> ds \\
   &+\varepsilon^2\int_0^t E |\sigma(s,X^{\varepsilon}(s))|^2ds\\
	  & \leq 2\int_0^t E\sqrt{|X^\varepsilon(s)-x(s)|^2  |b(s,X^{\varepsilon}(s))-b(s,x(s))|^2} ds\\
   &+ \varepsilon^2\int_0^t E |\sigma(s,X^{\varepsilon}(s))|^2ds \\
	  & \leq 2\int_0^t E\sqrt{|X^\varepsilon(s)-x(s)|^2 L_T^2|X^\varepsilon(s)-x(s)|^2} ds\\
   & + \varepsilon^2 K_T^2 \int_0^t  (1+E|X^\varepsilon(s)|^2)ds \\
	  & \leq  2L_T\int_0^t E|X^\varepsilon(s)-x(s)|^2ds+ \varepsilon^2 K_T^2\int_0^t(1+E|X^\varepsilon(s)|^2)ds
	\end{align*}

By Lemma \eqref{Gronwall}, choosing $m(t)=E|X^\varepsilon(t)-x(t)|^2$, $\alpha= 2L_T$, $C=\varepsilon^2 K_T^2\int_0^t(1+E|X^\varepsilon(s)|^2)ds$, we derive
	the following inequality
	
	\begin{align*}
	E|X^\varepsilon(t)-x(t)|^2
	&\leq e^{2L_{T}t} \varepsilon^2 K_T^2\int_0^t(1+E|X^\varepsilon(s)|^2)ds\\
	& \leq e^{2L_{T}t} \varepsilon^2 K_T^2 \int_{0}^{t}  (1+E|x_0|^2) \exp[(2K_T+\varepsilon^2K_T^2)s] ds \\
	& \leq \varepsilon^2 a(t)\,,
	\end{align*}
	and we conclude choosing $a(t)$
such that it is a monotone increasing function.\\
Concerning the second part of the theorem, we first use the definition of our diffusion processes \eqref{ode} and \eqref{perturbed sde}  to get
	\begin{align}
	\max_{0 \leq s \leq t}|X^\varepsilon(s)-x(s)|
	& =\max_{0 \leq s \leq t}|\int_{0}^{s} b(v,X^\varepsilon(v))dv + \varepsilon \int_{0}^{s}\sigma(v,X^\varepsilon(v))dw_v -  \int_{0}^{s}b(v,x(v))dv| \nonumber\\
	& \leq \max_{0 \leq s \leq t}|\int_{0}^{s} b(v,X^\varepsilon(v))dv-\int_0^s b(v,x(v)dv| \nonumber\\
  &+ \varepsilon \max_{0 \leq s \leq t} |\int_{0}^{s}\sigma(v,X^\varepsilon(v))dw_v | \nonumber\\
	& \leq \int_{0}^{t} |b(s,X^\varepsilon(s))-b(s,x(s))|ds +\varepsilon \max_{0 \leq s \leq t} |\int_{0}^{s}\sigma(v,X^\varepsilon(v))dw_v |. \label{(1.6)}
\end{align}
Then, by the Chebyshev inequality, setting $$\xi(\omega) = \int_{0}^{t} |b(s,X^\varepsilon(s))-b(s,x(s))|ds, \ a=\frac{\delta}{2}, \ f(x)=x^2\,,$$ 
we estimate the first term on the right hand side of \eqref{(1.6)}, i.e.:
	
	\begin{align}
	P\lbrace \int_{0}^{t}|b(s,X^\varepsilon(s))-b(s,x(s))|ds> \frac{\delta}{2} \rbrace 
	&\leq \frac{E[\int_{0}^{t}|b(s,X^\varepsilon(s))-b(s,x(s))|ds]^2}{(\frac{\delta}{2})^2} \nonumber \\
	& \leq 4 \delta^{-2} t E\int_{0}^{t} |b(s,X^\varepsilon(s))-b(s,x(s))|^2 ds \nonumber \\
	& \leq 4 t \delta^{-2} L_T^2  \int_{0}^{t}E|X^\varepsilon(s)-x(s)|^2 ds \nonumber \\
	& \leq 4t L_T^2 \delta^{-2} \varepsilon^2   \int_{0}^{t}a(s)ds \nonumber \\
	& =\varepsilon^2 \delta^{-2}a_1(t)
	 \label{(1.7)}
	\end{align}
	where in the second step we used the Cauchy-Schwarz inequality,  in the third we applied the Lipschitz condition for the drift, then concluding exploiting the first assertion of the theorem.
 Having obtained an estimate for the first term on the right hand side, we move forward to estimate the second term on the right hand side of \eqref{(1.6)}.
	To this aim we first use the generalized Kolmogorov inequality for stochastic integrals, namely:
	\begin{equation}\label{KolmogorovIntegral}
 P \lbrace \sup_{a \leq t \leq b} |\int_{a}^{t}f(s,\omega)dw_s|>C |\mathcal{N}_a \rbrace \leq \frac{1}{c^2} E(\int_{a}^{b} f^2(s,\omega)ds |\mathcal{N}_a) \,,
 \end{equation}
	then, setting $f(v,\omega)= \sigma(v,X^\varepsilon(v)), \ C=\frac{\delta}{2\varepsilon}, \mathcal{N}_a= \Omega, \ a=0, b=t$, we can rewrite \eqref{KolmogorovIntegral} as follows
	\begin{align}
	P \lbrace \varepsilon \max_{0 \leq s \leq t} |\int_{0}^{s} \sigma(v,X^\varepsilon(v))dw_v|>\frac{\delta}{2} \rbrace \nonumber
	& =	P \lbrace \sup_{0 \leq s \leq t} |\int_{0}^{s} \sigma(v,X^\varepsilon(v))dw_v|>\frac{\delta}{2 \epsilon} \rbrace \nonumber \\
	& \leq \frac{1}{(\frac{\delta}{2 \varepsilon})^2} E(\int_{0}^{t}|\sigma(s,X^\varepsilon(s))|^2ds) \nonumber \\
	& \leq \frac{4 \varepsilon^2}{\delta^2} \int_{0}^{t} K_T^2(1+E|X^\varepsilon(s)|^2)ds\nonumber \\
	& \leq 4 \varepsilon^2 \delta^{-2}K_T^2(1+E|x_0|^2)\int_0^t \exp[(2K_T+\varepsilon^2K_T^2)s]ds \nonumber \\
	&= \varepsilon^2 \delta^{-2} a_2(t)\, \label{(1.8)}
	\end{align}
	so that, combining  \eqref{(1.6)}, \eqref{(1.7)} and \eqref{(1.8)}, 	we have:
	\begin{align*}
	\lim_{\varepsilon \to 0}P\lbrace \max_{0 \leq s \leq t}|X^\varepsilon(s)-x(s)|>\delta\rbrace
	& \leq \lim_{\varepsilon \to 0} P \lbrace \int_{0}^{t}|b(s,X^\varepsilon(s))-b(s,x(s))|ds\\
	&+ \varepsilon\max_{0 \leq s \leq t}   |\int_{0}^{s}\sigma(v,X^\varepsilon(v))dw_v|>\delta \rbrace \\
	& \leq \lim_{\varepsilon \to 0} P \lbrace \int_{0}^{t}|b(s,X^\varepsilon(s))-b(s,x(s))|ds> \frac{\delta}{2}\rbrace \\
	&+ \lim_{\varepsilon \to 0} P \lbrace  \varepsilon\max_{0 \leq s \leq t}   |\int_{0}^{s}\sigma(v,X^\varepsilon(v))dw_v|>\frac{\delta}{2} \rbrace \\
	& \leq \lim_{\varepsilon \to 0} [\varepsilon^2 \delta^{-2}a_1(t)+ \varepsilon^2\delta^{-2}a_2(t)]=0.
	\end{align*}
 which concludes the proof.
\end{proof}	

\section{Zeroth Order Approximation for Dissipative Case} 
\subsection{Existence and Uniqueness of a Solution of the SDE with Dissipative Drift}
We aim at generalizing what obtained within the previous section to consider weaker conditions on the coefficients of eq. \eqref{perturbed sde}. Namely, we aim at finding that eq. \eqref{perturbed sde} still admits a unique solution when considering dissipativity for the drift and local Lipschitz  conditions for  all the other coefficients.
Moreover, after proving both existence and uniqueness, we will also derive  the zeroth order approximation for the process solving \eqref{perturbed sde}, under the same aforementioned assumptions on the drift as well as for the other coefficients.
Let us start proving a result similar to Theorem 2, Chapter 2, Section 6,  \cite{GS 1}, 
where the Authors considered  a similar SDE-problem in dimension $r=1$. In what follows, we will use it to derive our existence result, see Theorem \eqref{Existence, Uniqueness}.
\begin{Theorem}
	\label{Theorem 2 GS}
	Assume that the coefficients $b_1(t,x)$, $b_2(t,x)$, $\sigma_1(t,x)$, $\sigma_2(t,x)$ of the equations
	\begin{equation}
	\label{(4)}
	d{X}_{i}^\varepsilon(t) =b_i(t,X_{i}^\varepsilon(t))dt + \varepsilon \sigma_i(t,X_{i}^\varepsilon(t))dW_t, \ i=1,2
	\end{equation}
satisfy Lipschitz condition and linear growth condition, i.e.,
for every $T>0$ there exist constants $K_T,L_T$ such that for $t \in [0,T]$, $x,y \in \mathbb{R}^r$, $i=1,2$
\[
|b_i(t,x)-b_i(t,y)|+|\sigma_i(t,x)-\sigma_i(t,y)|\leq L_T|x-y|,
\]
\[
|b_i(t,x)|^2+|\sigma_i(t,x)|^2 \leq K_T^2(1+|x|^2)
\]
and that for some $N>0$ with $|x|\leq N$ ,
$b_1(t,x)=b_2(t,x)$ and $\sigma_1(t,x)=\sigma_2(t,x)$.\\
If $X_{1}^\varepsilon(t)$ and $X_{2}^\varepsilon(t)$ are solutions of \eqref{(4)} with the same initial condition $X_{1}^\varepsilon(0)= X_{2}^\varepsilon(0)=x_0$, $M[x_0^2] < \infty$, and $\tau_i$ is the largest t for which 
$\sup_{0 \leq s \leq t } |X_{i}^{\varepsilon}(s)|\leq N$, then $P \lbrace \tau_1=\tau_2 \rbrace=1$ and $$P \lbrace \sup_{0 \leq s \leq \tau_1} |X_{1}^\varepsilon(s)-X_{2}^\varepsilon(s)|=0 \rbrace =1.$$
\end{Theorem}

\begin{proof}
Define $\gamma_1(t):= 1$, if $\sup_{0 \leq s \leq t} |X_{1}^{\varepsilon}(s)|\leq N,$ and $\gamma_1(t):=0$ otherwise.\\
Then we get
\begin{align*}
\gamma_1(t)  [X_{1}^{\varepsilon}(t)-X_{2}^{\varepsilon}(t)] 
&= \gamma_1(t) \int_{0}^{t} [ b_1(s,X_{1}^\varepsilon(s))-b_2(s,X_{2}^\varepsilon(s))] ds \\
&+ \gamma_1(t) \varepsilon \int_{0}^{t} [ \sigma_1(s,X_{1}^\varepsilon(s))-\sigma_2(s,X_{2}^\varepsilon(s))] dW_s  \\
&= \gamma_1(t) \int_{0}^{t} [ b_1(s,X_{1}^\varepsilon(s))-b_2(s,X_{1}^\varepsilon(s))] ds  \\
&+ \gamma_1(t) \varepsilon \int_{0}^{t} [ \sigma_1(s,X_{1}^\varepsilon(s))-\sigma_2(s,X_{1}^\varepsilon(s))] dW_s \\
&+ \gamma_1(t) \int_{0}^{t} [ b_2(s,X_{1}^\varepsilon(s))-b_2(s,X_{2}^\varepsilon(s))] ds  \\
&+ \gamma_1(t) \varepsilon \int_{0}^{t} [ \sigma_2(s,X_{1}^\varepsilon(s))-\sigma_2(s,X_{2}^\varepsilon(s))] dW_s \\
&= \gamma_1(t) \int_{0}^{t} [ b_2(s,X_{1}^\varepsilon(s))-b_2(s,X_{2}^\varepsilon(s))] ds  \\
&+ \gamma_1(t) \varepsilon  \int_{0}^{t} [ \sigma_1(s,X_{1}^\varepsilon(s))-\sigma_2(s,X_{1}^\varepsilon(s))] dW_s  .
\end{align*}
Thus 
\begin{align*}
\gamma_1(t)[ X_{1}^{\varepsilon}(t)- X_{2}^{\varepsilon}(t) ]^2 
& \leq 2 \gamma_1(t) \Big [\int_{0}^{t} [ b_2(s,X_{1}^\varepsilon(s))-b_2(s,X_{2}^\varepsilon(s))] ds \Big ]^2  \\
&+ 2 \gamma_1(t) \varepsilon^2 \Big [ \int_{0}^{t} [ \sigma_1(s,X_{1}^\varepsilon(s))-\sigma_2(s,X_{1}^\varepsilon(s))] dW_s  \Big ]^2.
\end{align*}
Taking into account that $\gamma_1(t)=1$ implies $\gamma_1(s)=1$ for $s \leq t$ we can write the $\gamma_1(s)$ inside the brackets. Taking the expectation and then using the Lipschitz condition and Cauchy-Schwarz-inequality, we can show that for $\varepsilon \leq 1$ there exists a constant $l$ for which
\begin{align*}
E \Big [ \gamma_1(t)[ X_{1}^{\varepsilon}(t)- X_{2}^{\varepsilon}(t) ]^2  \Big ]
& \leq E \Big [4 [\int_{0}^{t} \gamma_1(s) L_T |X_{1}^\varepsilon(s)-X_{2}^\varepsilon(s)|ds]^2 \Big ] \\
& \leq 4 L_T^2 t \int_{0}^{t} E [\gamma_1(s)|X_{1}^\varepsilon(s)-X_{2}^\varepsilon(s)|^2]ds\\
& = l \int_{0}^{t} E[\gamma_1(s)|X_{1}^\varepsilon(s)-X_{2}^\varepsilon(s)|^2]ds\,,
\end{align*}
then, by Lemma \ref{Gronwall-Lemma}, for which C=0,
we have:

$$E\Big[\gamma_1(t)|X_{1}^\varepsilon(t)-X_{2}^\varepsilon(t)|^2\Big]=0.$$

Considering the continuity of $X_{1}^\varepsilon(t)$ and $X_{2}^\varepsilon(t)$ we can then establish
$$ P \lbrace \sup_{0 \leq t \leq T} \gamma_1(t)  |X_{1}^{\varepsilon}(t)-X_{2}^{\varepsilon}(t)|^2 =0 \rbrace =1 $$
On the interval $[0,\tau_1]$ the processes $X_1^\varepsilon(t)$ and $X_2^\varepsilon(t)$ coincide with probability 1. Hence $P \lbrace \tau_2 \geq \tau_1 \rbrace =1.$
Interchanging the indices 1 and 2 in the proof of the theorem, we analogously show that $P \lbrace \tau_1 \geq \tau_2 \rbrace =1.$
\end{proof}

\begin{Theorem}
	\label{Existence, Uniqueness}
	Let the coefficients of \eqref{perturbed sde} be defined and measurable for $t \in [0,T]$, and satisfy the conditions: 
	\begin{enumerate}
		\item  for every $T>0$, there exists a constant $K_T$ such that for all $(t,x)\in [0,T]\times\mathbb{R}^r$
	\begin{equation}
	\label{dissipativity}
	<x,b(t,x)> + |\sigma(t,x)|^2 \leq K_T^2(1+ |x|^2);
	\end{equation}
		\item 
		for every N and $T>0$ there exists an $L_{N,T}$ for which
		\begin{equation}
		\label{local Lipschitz}
		|b(t,x)-b(t,y)| + |\sigma(t,x)-\sigma(t,y)| \leq L_{N,T}|x-y|
		\end{equation}
		with $|x| \leq N$, $|y| \leq N$.
	\end{enumerate} 
Then \eqref{perturbed sde} has a unique solution in the sense that for two solutions $X_1^{\varepsilon}(t)$ and $X_2^{\varepsilon}(t)$ it holds
$$  P\lbrace \sup_{0 \leq t \leq T} |X_1^{\varepsilon}(t)- X_2^{\varepsilon}(t)|=0  \rbrace =1. $$
\end{Theorem}

\begin{proof} 
	We will first start showing the existence of a solution, then  we prove its uniqueness. Let us define  
	\[
	b_N(t,x)=
	\begin{cases}
	b(t,x) \quad if \quad |x|\leq N\\
	b(t,Nsign(x)) \quad if \quad |x|>N
	\end{cases}
	\]
$\sigma_N(t,x)$ is defined similarly.
	By $X_N^\varepsilon(t)$ we denote the solution of
	\begin{equation}
	\label{approximate-diffusion}
	dX_N^\varepsilon(t) =b_N(t,X_N^\varepsilon(t))+ \varepsilon \sigma_N(t,X_N^\varepsilon(t))dW_t, \ X_N^\varepsilon (0)=x_0\,.
	\end{equation}
For the latter equation all conditions for existence and uniqueness are given, since can count on a growth condition and global Lipschitz condition.
Let $\tau_N$ be the largest value of t for which $\sup_{0 \leq s \leq t}|X_N^\varepsilon(s)| \leq N$. Let $N^{'} >N$. Since $b_N(t,x)=b_{N^{'}}(t,x)$ and $\sigma_N(t,x)=\sigma_{N^{'}}(t,x)$ for all $|x|\leq N$ ,now we apply Theorem \ref{Theorem 2 GS} and obtain
$X_N^\varepsilon(t)=X_{N^{'}}^\varepsilon(t)$, with probability 1 for $t \in [0,\tau_N]$.\\
Hence for $N^{'}>N$:
	$$ P \lbrace \sup_{0 \leq t \leq T}|X_N^\varepsilon(t)-X_{N^{'}}^\varepsilon(t)|>0 \rbrace \leq P \lbrace \tau_N > T \rbrace= P \lbrace \sup_{0 \leq t \leq T}|X_N^\varepsilon(t)|>N \rbrace. $$
	
If we show that the probability on the right hand side converges to zero for $N \rightarrow \infty$, then it will  follow that $X_N^\varepsilon$ converges uniformly with probability 1 to some limit $X^\varepsilon(t)$ as $N \rightarrow \infty$.\\
Going to the limit in 
	$$
	X_{N}^\varepsilon(t) =x_0+ \int_{0}^{t} b_N(s,X_{N}^\varepsilon(s))ds+ \varepsilon \int_{0}^{t} \sigma_N(s,X_{N}^\varepsilon(s)) dW_s$$
	we convince ourselves that $X^\varepsilon(t)$ is equal with probability 1 to the continuous solution of \eqref{perturbed sde}. Therefore,	 to finish the proof of the existence of a solution, we are left with proving that: 

\begin{equation}
\label{probability N}
\lim_{N \rightarrow \infty}
P \lbrace \sup_{0 \leq t \leq T}|X_N^\varepsilon(t)|>N \rbrace=0\,,
\end{equation}
and to this aim, we first define the function $\psi(y)=\frac{1}{1+|y|^2}$, to then use the Ito formula.
We obtain
\begin{align*}
	&E[|X_N^\varepsilon(t)|^2 \psi(x_0)]-E[|x_0|^2\psi(x_0)] \\
	=& E\Big \lbrack  \int_{0}^{t} [2 \psi(x_0) <X_N^\varepsilon(s),b(t,X_N^\varepsilon(s))> + \varepsilon^2 \psi(x_0)  [\sigma_N(s,X_N^\varepsilon(s)]^2 ]ds \Big \rbrack   \\
	 \leq & E\Big \lbrack \psi(x_0) \int_{0}^{t} [2K_T^2(1+|X_N^\varepsilon(s)|^2) + \varepsilon^2  K_T^2(1+|X_N^\varepsilon(s)|^2)] ds \Big \rbrack   \\
	 \leq &   (2K_T^2t+\varepsilon^2K_T^2 t)E\psi(x_0)+ (2K_T^2+\varepsilon^2K_T^2)\int_{0}^{t} E[\psi(x_0) |X_N^\varepsilon(s)|^2] ds 	
\end{align*}	 
	By having that we can use Lemma \ref{Gronwall-Lemma} to get
	$$	E[\psi (x_0)|X_N^\varepsilon(t)|^2] 
	\leq [(2K_T^2t+\varepsilon^2K_T^2t)E\psi(x_0)+E(|x_0|^2 \psi(x_0)) ]e^{(2K_T^2+ \varepsilon^2K_T^2)t}.$$
	Which means we have 
	$$E[\psi (x_0) \sup_{0 \leq t \leq T}|X_N^\varepsilon(t)|^2] \leq C_1$$
where $C_1$ is independent of N.\\
We can moreover write
\begin{align*}
P \lbrace \sup_{0 \leq t \leq T} |X_N^\varepsilon(t)|>N \rbrace 
&= P \lbrace \psi(x_0)\sup_{0 \leq t \leq T} |X_N^\varepsilon(t)|^2>N^2 \psi(x_0) \rbrace \\
&\leq P \lbrace \psi(x_0)\sup_{0 \leq t \leq T} |X_N^\varepsilon(t)|^2> \delta N^2  \rbrace + P \lbrace \psi(x_0) \leq \delta \rbrace \\
& \leq \frac{C_1}{\delta N^2}+  P \lbrace \psi(x_0) \leq \delta \rbrace,
\end{align*}
where the last inequality follows from the Chebychev one. \\
Consequently
$$\overline{\lim\limits_{N \rightarrow \infty}} P \lbrace \sup_{0 \leq t \leq T} |X_N^\varepsilon(t)|>N \rbrace \leq P \lbrace \psi (x_0) \leq \delta \rbrace. $$	
Since $\delta$ is an arbitrary positive number and
\[
P \lbrace \psi(x_0)=0 \rbrace =0,
\]
\eqref{probability N}
results from the preceding relation, which completes the proof of the existence of a solution to \eqref{perturbed sde}.\\
Next we want to prove the uniqueness of the solution.
Let $X_1^\varepsilon(t)$ and $X_2^\varepsilon(t)$ be two solutions of \eqref{perturbed sde}. Denoting by $\phi(t)$ the variable equal to 1 if 
$\sup_{0 \leq s \leq t}|X_1^\varepsilon(s)|\leq N$ and $\sup_{0 \leq s \leq t}|X_2^\varepsilon(s)|\leq N$ and equal to 0 otherwise, then by the  second condition, we have:
\begin{align*}
E|X_1^{\varepsilon}(t)-X_2^\varepsilon(t)|^2 \phi(t)
& \leq 2 E[\phi(t)(\int_{0}^{t} |b(s,X_1^{\varepsilon}(s))-b(s,X_2^{\varepsilon}(s))| ds)^2] \\
& + 2 E[\phi(t)(\int_{0}^{t} |\sigma(s,X_1^{\varepsilon}(s))-\sigma(s,X_2^{\varepsilon}(s))| dW_s)^2] \\
& \leq  2t E[(\int_{0}^{t} \phi(s) |b(s,X_1^{\varepsilon}(s))-b(s,X_2^{\varepsilon}(s))|^2 ds)] \\
& + 2 E[(\int_{0}^{t} \phi(s) |\sigma(s,X_1^{\varepsilon}(s))-\sigma(s,X_2^{\varepsilon}(s))|^2 ds)] \\
& \leq (2T+2) L_{N,T}^2 \int_{0}^{t} E(\phi(s)|X_1^\varepsilon(s)-X_2^\varepsilon(s)|^2) ds\,,
\end{align*}
where we  used that $(a-b)^2 \leq 2a^2+2b^2$, then we applied the Cauchy-Schwarz inequality and the Itô integral properties, afterwards exploiting the local Lipschitz continuity.
Then, by Lemma \ref{Gronwall-Lemma} with $C=0$, we obtain 
$E|X_1^{\varepsilon}(t)-X_2^\varepsilon(t)|^2 \phi(t)=0$, implying:
$$P\lbrace X_1^{\varepsilon}(t) \neq X_2^{\varepsilon}(t) \rbrace \leq 
P\lbrace \sup_{0 \leq t \leq T} |X_1^{\varepsilon}(t)|>N \rbrace + P \lbrace \sup_{0 \leq t \leq T} |X_2^{\varepsilon}(t)|>N \rbrace\,,  $$
indeed,   $\phi (t)$ is zero for 
$\sup_{0 \leq s \leq t}|X_1^{\varepsilon}(s)| >N$ or $\sup_{0 \leq s \leq t}|X_2^{\varepsilon}(s)|> N$.
 Hence the probability on the right side of this inequality tend to zero as N $\rightarrow \infty$, i.e., for all $t \in [0,T]$: $P \lbrace X_1^\varepsilon(t)= X_2^\varepsilon(t) \rbrace =1 $, implying uniqueness in the sense that 
   $P\lbrace \sup_{0 \leq t \leq T} |X_1^{\varepsilon}(t)- X_2^{\varepsilon}(t)|=0  \rbrace =1. $
\end{proof}

\subsection{Zeroth Order Approximation for Dissipative Case}
Considering the existence and uniqueness result proved in the  previous section  we aim to prove that the solution of the perturbed SDE \eqref{perturbed sde}  converges to the solution of the unperturbed one \eqref{ode} as the perturbation parameter $\varepsilon \to 0$, under dissipativity and dissipativity for differences for the drift vector and the local Lipschitz condition for all coefficients.

\begin{Theorem} \label{Theorem 1.2.2}
	Assume that the coefficients of \eqref{perturbed sde} satisfy a local Lipschitz condition, $\sigma$ increases no faster than linearly and b satisfies dissipativity and dissipativity for the differences:
	
		\begin{enumerate}
		\item  for every $T>0$ there exists a constant $K_T$ such that for all $x,y \in \mathbb{R}^r$ and $t \in [0,T]$
		\begin{align}
		\label{dissipativity 2}
		<x,b(t,x)> + |\sigma(t,x)|^2 &\leq K_T^2(1+ |x|^2);
		\\
		<x-y,b(t,x)-b(t,y)> &\leq K_T^2(1+ |x-y|^2);
		\label{dissipativity for differences}
		\end{align}
		\item 
		for every N and $T>0$ there exists an $L_{N,T}$ for which
		\begin{equation}
		\label{local Lipschitz 2}
		|b(t,x)-b(t,y)| + |\sigma(t,x)-\sigma(t,y)| \leq L_{N,T}|x-y|
		\end{equation}
		with $|x| \leq N$, $|y| \leq N$.
	\end{enumerate} 
	
	Then for all $t>0$ and $\delta>0$ we have:
	\begin{enumerate}
		\item $E|X^\varepsilon(t)-x(t)|^2\leq \varepsilon^2 a(t)$, and
		\item $\lim_{\varepsilon \to 0} P\lbrace \max_{0 \leq s \leq t} |X^\varepsilon(s) -x(s)|>  \delta \rbrace=0$
	\end{enumerate}
where a(t) is a monotone increasing function expressed in terms of $E|x_0|$, $L_{N,T}$ and $K_T$.
	
\end{Theorem}

\begin{proof}
	We start showing that $E|X^\varepsilon(t)|^2$ is bounded uniformly in $\varepsilon \in [0,1]$.
	Applying Ito's formula  and mathematical expectation we get
		\[
	  1+ E|X^\varepsilon(t)|^2 
	  = 1+E|x_0|^2+2\int_0^t E < X^\varepsilon(s),b(s,X^{\varepsilon}(s))> ds +
	   \varepsilon^2\int_0^t E |\sigma(s,X^{\varepsilon}(s))|^2ds
	   \]
	
Using the fact that $\sigma$  increases no faster than linearly and the dissipativity for b, the last relation implies the estimate
	\begin{align*}
	1+ E|X^\varepsilon(t)|^2 
	  &= 1+E|x_0|^2+2\int_0^t E < X^\varepsilon(s),b(s,X^{\varepsilon}(s))> ds +
	   \varepsilon^2\int_0^t E |\sigma(s,X^{\varepsilon}(s))|^2ds \\
	& \leq 1+ E|x_0|^2+ 2 \int_{0}^{t} E [K_T^2(1+|X^\varepsilon(s)|^2)] ds
	+\varepsilon^2 \int_{0}^{t} E [K_T ^2(1+|X^\varepsilon(s)|^2)] ds\\
	& \leq  1+ E|x_0|^2+ (2 K_T^2 + \varepsilon^2 K_T^2) \int_{0}^{t} (1+E|X^\varepsilon(s)|^2)ds\,.
	\end{align*}
	
 Next we use Lemma \ref{Gronwall-Lemma}
	choosing $m(t)=1+M|X^\varepsilon(t)|^2$, $C=1+E|x_0|^2$ and $\alpha=(2K_T^2+\varepsilon^2K_T^2)$, to obtain:
	\begin{equation}
	\label{(1.5.2)}
	1+ E|X^\varepsilon(t)|^2 \leq (1+E|x_0|^2) \exp[(2K_T^2+\varepsilon^2K_T^2)t].
	\end{equation}
	
	Inequality \eqref{(1.5.2)} proves that $E|X^\varepsilon(t)|^2$ is bounded uniformly in $\varepsilon \in [0,1]$ and by previously stated result we  show, proceeding as before, that $E|X^\varepsilon(t)-x(t)|^2 \leq \varepsilon^2 a(t)$. Indeed, applying the Ito's formula to the function $|X^\varepsilon(t)-x(t)|^2$, then taking the mathematical expectation on both sides of the equality, we get:
 \begin{align*}
     E|X^\varepsilon(t)-x(t)|^2 &= 2 \int_{0}^{t} E <X^\varepsilon(s)-x(s), b(s,X^\varepsilon(s))-b(s,x(s))>ds \\
     & + \varepsilon^2\int_{0}^{t}E |\sigma(s,X^\varepsilon(s))|^2ds.
 \end{align*}

In the proof of the existence we proved \eqref{probability N}. Since $X_{N}^{\varepsilon}(t)$ converges to $X^\varepsilon(t)$ as $N \to \infty$, moreover we know that:
\begin{equation}
\lim\limits_{N \to \infty} P \lbrace \sup_{0 \leq s \leq T} |X^\varepsilon(s)| >N  \rbrace =0
\end{equation}
and
\begin{equation}
\lim\limits_{N \to \infty} P \lbrace \sup_{0 \leq s \leq T} |x(s)| >N  \rbrace =0\,.
\end{equation}
The latter implies that there exists an N, such that 
\begin{equation}
\label{prob1}
P \lbrace \sup_{0 \leq s \leq T} |X^\varepsilon(s)| >N  \rbrace \leq \frac{\varepsilon^2}{2}
\end{equation}
and 
\begin{equation}
\label{prob2}
P \lbrace \sup_{0 \leq s \leq T} |x(s)| >N  \rbrace \leq \frac{\varepsilon^2}{2}.
\end{equation}

In the following calculations we first split up our mathematical expectation in two different cases, then we use the Cauchy-Schwarz inequality and condition \eqref{dissipativity 2}.
Afterwards, we apply the local Lipschitz condition \eqref{local Lipschitz 2} and the dissipativity for the differences for the drift coefficient \eqref{dissipativity for differences}. Then we use inequalities \eqref{prob1} and \eqref{prob2} for the  probabilities:
\begin{align*} 
    \scalemath{0.85}{E|X^\varepsilon(t)-x(t)|^2} &= \scalemath{0.85}{ 2 \int_{0}^{t} E <X^\varepsilon(s)-x(s), b(s,X^\varepsilon(s))-b(s,x(s))>ds}\\
	&+\scalemath{0.85}{\varepsilon^2\int_{0}^{t}E |\sigma(s,X^\varepsilon(s))|^2ds}\\ 
	&\leq \scalemath{0.85}{P \lbrace \max \lbrace \sup_{0 \leq s \leq T} |X^\varepsilon(s)|, \sup_{0 \leq s \leq T}|x(s)| \rbrace \leq N \rbrace \times} \\
	& \scalemath{0.78}{2 \int_{0}^{t} E \Big [ \sqrt{|X^\varepsilon(s)-x(s)|^2 |b(s,X^\varepsilon(s))-b(s,x(s))|^2}ds \Big |
	 \max \lbrace \sup_{0 \leq s \leq T} |X^\varepsilon(s)|, \sup_{0 \leq s \leq T}|x(s)| \rbrace \leq N
	  \Big ]}\\
	&+ \scalemath{0.85}{ P \lbrace \max \lbrace \sup_{0 \leq s \leq T} |X^\varepsilon(s)|, \sup_{0 \leq s \leq T}|x(s)| \rbrace  > N \rbrace \times} \\
	& \scalemath{0.78}{2 \int_{0}^{t}
	E\Big [<X^\varepsilon(s)-x(s), b(s,X^\varepsilon(s))-b(s,x(s))> ds \Big |\max \lbrace \sup_{0 \leq s \leq T} |X^\varepsilon(s)|, \sup_{0 \leq s \leq T}|x(s)| \rbrace > N \Big ]}
	 \\ 
	&+ \scalemath{0.85}{\varepsilon^2 K_T^2 \int_{0}^{t}(1+E|X^\varepsilon(s)|^2) ds} \\	
	& \scalemath{0.78}{\leq P \lbrace \max \lbrace \sup_{0 \leq s \leq T} |X^\varepsilon(s)|, \sup_{0 \leq s \leq T}|x(s)| \rbrace \leq N \rbrace
	2 \int_{0}^{t} E \sqrt{|X^\varepsilon(s)-x(s)|^2 L_{N,T}^2 |X^\varepsilon(s)-x(s)|^2}ds} \\
	&+\scalemath{0.85}{ P \lbrace \max \lbrace \sup_{0 \leq s \leq T} |X^\varepsilon(s)|, \sup_{0 \leq s \leq T}|x(s)| \rbrace > N \rbrace
	2 \int_{0}^{t} K^2(1+E|X^\varepsilon(s)-x(s)|^2)ds} \\
	&+\scalemath{0.85}{\varepsilon^2 K_T^2 \int_{0}^{t}(1+E|X^\varepsilon(s)|^2) ds} \\ 
	&\leq
	\scalemath{0.85}{2 L_{N,T} \int_{0}^{t} E |X^\varepsilon(s)-x(s)|^2 ds} \\
	
	&+ \scalemath{0.78}{(P \lbrace \sup_{0 \leq s \leq T} |X^\varepsilon(s)|> N \rbrace+
	P \lbrace \sup_{0 \leq s \leq T} |x(s)|> N \rbrace)
	[2 K_T^2t+ 2 K_T^2 \int_{0}^{t} E|X^\varepsilon(s)-x(s)|^2 ds]} \\
	&+ \scalemath{0.85}{\varepsilon^2 K_T^2 \int_{0}^{t}(1+E|X^\varepsilon(s)|^2)ds}  \\ 
	&\leq
	\scalemath{0.85}{2 L_{N,T} \int_{0}^{t} E|X^\varepsilon(s)-x(s)|^2ds + \varepsilon^2
	2 K_T^2t}\\
 &+\scalemath{0.85}{ \varepsilon^2 2 K_T^2 \int_{0}^{t} E|X^\varepsilon(s)-x(s)|^2ds 
	+\varepsilon^2 K_T^2 \int_{0}^{t}(1+E|X^\varepsilon(s)|^2)ds} \\
	&\leq
	\scalemath{0.85}{(2 L_{N,T}+\varepsilon^2 2K_T^2) \int_{0}^{t}E|X^\varepsilon(s)-x(s)|^2 ds} \\
 &+\scalemath{0.85}{\varepsilon^2
	2 K_T^2t +\varepsilon^2 K_T^2 \int_{0}^{t}(1+E|X^\varepsilon(s)|^2) ds}.
\end{align*}
	
	By lemma \ref{Gronwall-Lemma},  choosing $m(t)=E|X^\varepsilon(t)-x(t)|^2, \ \alpha=2 L_{N,T}+2\varepsilon^2 K_T^2, \ C= \varepsilon^2
	2 K_T^2t +\varepsilon^2 K_T^2 \int_{0}^{t}(1+E|X^\varepsilon(s)|^2) ds$, we have:  
	\begin{align*}
	E|X^\varepsilon(t)-x(t)|^2 
	&\leq e^{(2 L_{N,T}+2\varepsilon^2 K_T^2)t} 
	[\varepsilon^2
	2 K_T^2t +\varepsilon^2 K_T^2 \int_{0}^{t}(1+E|X^\varepsilon(s)|^2) ds]\\
	& \leq e^{(2 L_{N,T}+2\varepsilon^2 K_T^2)t}\varepsilon^2
	2 K_T^2t\\
	& + e^{(2 L_{N,T}+2\varepsilon^2 K_T^2)t}  \varepsilon^2 K_T^2 \int_{0}^{t} (1+E|x_0|^2) \exp[(2K_T^2+\varepsilon^2K_T^2)s] ds \\
	& \leq \varepsilon^2 2 K_T^2 t e^{(2 L_{N,T}+2\varepsilon^2 K_T^2)t} \\
	& + \varepsilon^2 K_T^2 e^{(2 L_{N,T}+2\varepsilon^2 K_T^2)t}
	(1+E|x_0|^2) \int_{0}^{t}  \exp[(2K_T^2+\varepsilon^2K_T^2)s]ds \\
	& \leq \varepsilon^2 a(t).
\end{align*}
	
	Where we used the result \eqref{(1.5.2)} and $a(t)$ has been chosen such that it is a monotone  increasing function.

Now we want to prove the second assertion of the theorem.
We will now use the Chebyshev inequality according which $P\lbrace \xi(\omega) \geq a \rbrace \leq \frac{Ef(\xi)}{f(a)}$.

	By setting $\xi(\omega) = \max_{0 \leq s \leq t}|X^\varepsilon(s)-x(s)|, \ a=\delta, \ f(x)=x^2$ and  applying the first assertion of the theorem we obtain

\begin{align}
	\label{eq Cheb}
	P\lbrace \max_{0 \leq s \leq t}|X^\varepsilon(s)-x(s)|> \delta \rbrace 
	&\leq \frac{E[\max_{0 \leq s \leq t}|X_s^\epsilon-x_s|]^2}{\delta^2} \nonumber \\
		& \leq \frac{ \varepsilon^2 a(t)}{\delta^2}
	   \nonumber  \\ 
\end{align}
	Taking limits on both sides in \eqref{eq Cheb} , we get
\[
	\lim_{\varepsilon \to 0} P\lbrace \max_{0 \leq s \leq t} |X^\varepsilon(s)-x(s)|>  \delta \rbrace \leq \lim_{\varepsilon \to 0} \frac{ \epsilon^2 a(t)}{\delta^2}=0
\]
\end{proof}	

\subsection{Parabolic Differential equations with a Small Parameter: Cauchy Problem, Transport Equation}

We aim at deriving results concerning the behavior of solutions of Cauchy problem as $\varepsilon \to 0$, from the behavior of $X^\varepsilon(t)$ as $\varepsilon \to 0$. Within preceding section we have obtained a result concerning the behavior of solutions  $X^\varepsilon(t)$ as $\varepsilon \to 0$, which we will exploit in what follows.
In particular, we consider the Cauchy problem:
\begin{equation}
\label{(3.2.2)}
\frac{\partial v^\varepsilon(t,x)}{\partial t} = L^\varepsilon v^\varepsilon(t,x) + c(x) v^\varepsilon (t,x) +g(x), \ v^\varepsilon(0,x)=f(x),
\end{equation}
$t>0, \ x \in \mathbb{R}^r$ for $\varepsilon>0$, together with the problem for the first-order operator obtained considering $\varepsilon=0$:
\begin{equation}
\label{(3.3.2)}
\frac{\partial v^0(t,x)}{\partial t} = L^0 v^0(t,x) + c(x) v^0 (t,x) +g(x), \ v^0(0,x)=f(x)\,,\end{equation}
where $L^\varepsilon$ is a differential operator with a small parameter at the derivatives of highest order:
$$ L^\varepsilon= \frac{\varepsilon^2}{2} \sum_{i,j=1}^{r} a^{ij}(t,x) \frac{\partial^2}{\partial x^i x^j}+ \sum_{i=1}^{r} b^i(t,x) \frac{\partial}{\partial x^i}\,.$$
Every  operator $L^\varepsilon$ (whose coefficients are assumed to be sufficiently regular) has an associated  diffusion process $X^{\varepsilon,x}(t)$.
This diffusion process can be represented by mean of a stochastic equation of the following type:
\begin{equation}
\label{(3.1)}
    dX^\varepsilon(t)= b(t,X^\varepsilon(t))dt + \varepsilon \sigma(t,X^\varepsilon(t)) dW_t, \ X^\varepsilon(0) =x
\end{equation}
where $\sigma(t,x)\sigma^{*}(t,x)=(a^{ij}(t,x)),\ b(t,x)=(b^1(t,x),...,b^r(t,x)).$\\
Solutions of partial differential equations(PDEs) can be represented  in terms of the solutions of stochastic differential equations, which is no one else but the well known Feyman-Kac formula. \\
Hence, assuming  the following conditions are satisfied:
\begin{enumerate}
	\item the function $c(x)$ is uniformly continuous and bounded for $x \in \mathbb{R}^r$;
	\item the coefficients of $L^\varepsilon$ satisfy a local Lipschitz condition, b satisfies dissipativity and dissipativity for the differences;

	\item $k^{-2} \sum\lambda_i^2 \leq \sum_{i,j=1}^r a^{ij}(t,x) \lambda_i \lambda_j \leq k^2 \sum \lambda_i^2$ for any real $\lambda_1,\lambda_2,...,\lambda_r$ \\ and $x \in \mathbb{R}^r$, $t \in [0,T]$ where $k^2$ is a positive constant.
\end{enumerate}
the solutions to problems \eqref{(3.2.2)} and \eqref{(3.3.2)}, both exist and are unique, moreover the solution $v^\varepsilon (t,x)$ of \eqref{(3.2.2)}  is given by

\begin{align}
\label{Feynman-Kac}
	v^\varepsilon(t,x)&= E^x[f(X^{\varepsilon,x}(t)) \exp\lbrace \int_{0}^{t}c(X^{\varepsilon,x}(s))ds \rbrace] \nonumber \\
 &+ E^x[\int_{0}^{t} g(X^{\varepsilon,x}(s)) \exp \lbrace \int_{0}^{s} c(X^{\varepsilon,x}(r))dr \rbrace ds]\,.
\end{align}
To prove formula \eqref{Feynman-Kac}, let us consider the stochastic process 
\begin{equation}
\label{stoch process}
    Y_s=exp\big\{\int_0^s c(X^{\varepsilon,x}(r))dr\big\}v^\varepsilon(t-s,X^{\varepsilon,x}(s))+\int_0^s g(X^{\varepsilon,x}(r))exp\big\{\int_0^r c(X^{\varepsilon,x}(\tau))d\tau\big\}dr\,,
\end{equation}
them applying the {\it product rule} to the process \eqref{stoch process}, along with the Ito's lemma, we get: 
\begin{align*}
 dY_s&=d \bigg( exp\big\{\int_0^s c(X^{\varepsilon,x}(r))dr\big\} \bigg)v^\varepsilon(t-s,X^{\varepsilon,x}(s))\\
 &+exp\big\{\int_0^s c(X^{\varepsilon,x}(r))dr\big\}dv^\varepsilon(t-s,X^{\varepsilon,x}(s))\\
 &+d \bigg( exp\big\{\int_0^s c(X^{\varepsilon,x}(r))dr\big\}\bigg)dv^\varepsilon(t-s,X^{\varepsilon,x}(s))\\
 &+d\bigg(\int_0^s g(X^{\varepsilon,x}(r))exp\big\{\int_0^r c(X^{\varepsilon,x}(\tau))d\tau\big\}dr \bigg) \\
& = c(X^{\varepsilon,x}(s))v^\varepsilon(t-s,X^{\varepsilon,x}(s))\big(exp\big\{\int_0^s c(X^{\varepsilon,x}(r))dr\big\}\big)ds\\
&+ exp\big\{\int_0^s c(X^{\varepsilon,x}(r))dr\big\}dv^\varepsilon(t-s,X^{\varepsilon,x}(s))\\
&+ g(X^{\varepsilon,x}(s))exp\big\{\int_0^s c(X^{\varepsilon,x}(\tau))d\tau\big\}ds\\
&=exp\big\{\int_0^s c(X^{\varepsilon,x}(r))dr\big\}\bigg(c(X^{\varepsilon,x}(s))v^\varepsilon(t-s,X^{\varepsilon,x}(s))\\
&+g(X^{\varepsilon,x}(s))-\frac{\partial v^\varepsilon}{\partial s}+ \frac{\varepsilon^2}{2} \sum_{i,j=1}^{r} a^{ij}(s,X) \frac{\partial^2 v^\varepsilon}{\partial X^i X^j}+ \sum_{i=1}^{r} b^i(s,X) \frac{\partial v^\varepsilon}{\partial X^i}\bigg)ds\\
&+exp\big\{\int_0^s c(X^{\varepsilon,x}(r))dr\big\}\varepsilon \sigma(s,X)\frac{\partial v^\varepsilon}{\partial X}dW_s=exp\big\{\int_0^s c(X^{\varepsilon,x}(r))dr\big\}\varepsilon \sigma(s,X)\frac{\partial v^\varepsilon}{\partial X}dW_s
\end{align*}
Since $v$ satisfies partial differential equation \eqref{(3.2.2)} we will have 
\[
dY_s=exp\big\{\int_0^s c(X^{\varepsilon,x}(r))dr\big\}\varepsilon \sigma(s,X)\frac{\partial v^\varepsilon}{\partial X}dW_s\,,
\]
therefore $Y_s$ is a martingale and it follows that
\begin{align*}
    v^\varepsilon(t,x)&=Y_0=E^x Y_t\\
    &= E^x \bigg[f(X^{\varepsilon,x}(t))exp\big\{\int_0^t c(X^{\varepsilon,x}(s))ds\big\}\\
    &+\int_0^t g(X^{\varepsilon,x}(s))exp\big\{\int_0^s c(X^{\varepsilon,x}(r))dr\big\}ds\bigg]\,.
\end{align*}
Using obtained formula for the solution of the PDE \eqref{(3.2.2)}, we have the following result:
	\begin{Theorem}
	\label{Theorem Cauchy Dissipativity}
	If conditions (1)-(3) are satisfied, then the limit $\lim\limits_{\varepsilon \to 0} v^\varepsilon(t,x)= v^0(t,x)$ exists for every bounded continuous initial function f(x), $x \in \mathbb{R}^r$. \\
	The function $v^0(t,x)$ is a solution of problem \eqref{(3.3.2)}.	
\end{Theorem}

\begin{proof}
	As proved above, the solution of \eqref{(3.2.2)} can be represented as:
	\begin{equation}
	\label{(3.4.2)}
	v^\varepsilon(t,x)=E^x [f(X^{\varepsilon,x}(t))exp\big\{\int_0^t c(X^{\varepsilon,x}(s))ds\big\}+\int_0^t g(X^{\varepsilon,x}(s))exp\big\{\int_0^s c(X^{\varepsilon,x}(r))dr\big\}ds]\,.
	\end{equation}
	
We use the convergence of $X^{\varepsilon,x}(s)$ to $X^{0,x}(s)$ in probability on the interval $[0,t]$ as $\varepsilon \to 0$ (Theorem \ref{Theorem 1.2.2}). Having under the sign of mathematical expectation in \eqref{(3.4.2)}  a bounded continuous functional of $X^{\varepsilon,x}(s)$, we can also use the Lebesgue dominated convergence theorem to derive:  

	\begin{align*}
	\lim\limits_{\varepsilon \to 0} v^\varepsilon(t,x) 
	&= \lim\limits_{\varepsilon \to 0} E^x[f(X^{\varepsilon,x}(t))exp\big\{\int_0^t c(X^{\varepsilon,x}(s))ds\big\}] \\
	&+ \lim\limits_{\varepsilon \to 0} E^x[\int_0^t g(X^{\varepsilon,x}(s))exp\big\{\int_0^s c(X^{\varepsilon,x}(r))dr\big\}ds] \\
	&= E^x[ \lim\limits_{\varepsilon \to 0} f(X^{\varepsilon,x}(t))exp\big\{\int_0^t c(X^{\varepsilon,x}(s))ds\big\}] \\
	&+ E^x[ \lim\limits_{\varepsilon \to 0} \int_{0}^{t} g(X^{\varepsilon,x}(s))exp\big\{\int_0^s c(X^{\varepsilon,x}(r))dr\big\}ds] \\
	&= E^x\bigg[f(X^{0,x}(t))exp\big\{\int_0^t c(X^{0,x}(s))ds\big\}\\
 &+\int_0^t g(X^{0,x}(s))exp\big\{\int_0^s c(X^{0,x}(r))dr\big\}ds\bigg]\,,
	\end{align*}
	the function on the right side of the equality is a  solution of \eqref{(3.3.2)}, hence concluding the proof.
\end{proof}

Considering $c(x) \equiv g(x) \equiv 0$ we obtain the special case of the transport equation, namely:
$$\frac{\partial v^\varepsilon(t,x)}{\partial t} = L^\varepsilon v^\varepsilon(t,x), \ v^\varepsilon(0,x)=f(x),$$
whose solution can be written in the following form:
\begin{equation}
	v^\varepsilon(t,x)= E^x[f(X^{\varepsilon,x}(t))] .
\end{equation}
Therefore, as in the case of the Cauchy problem, passing to the limit when $\varepsilon \to 0$, we get $\lim\limits_{\varepsilon \to 0} v^\varepsilon(t,x)= v^0(t,x)$ which represents the solution of 
$$\frac{\partial v^0(t,x)}{\partial t} = L^0 v^0(t,x), \ v^0(0,x)=f(x).$$

\section{Acknowledgments}

Yu.K. is very thankful to Mark Freidlin for an instructive discussion.

The financial support by the Ministry 
for Science and Education of Ukraine
through Project 0122U000048
is gratefully acknowledged.

\end{document}